\documentclass[10pt,leqno]{amsart}
\usepackage{amscd}
\usepackage{amssymb}
\usepackage{amsfonts}
\usepackage{latexsym}
\usepackage{verbatim}
\usepackage{amsthm}
\numberwithin{equation}{section}

\theoremstyle{plain}
\newtheorem{theo}{Theorem}[]
\newtheorem{theorem}{Theorem}[section]

\newtheorem{defn}[theorem]{Definition}

\newtheorem{prop}[theorem]{Proposition}
\newtheorem{cor}[theorem]{Corollary}
\newtheorem{rem}[theorem]{Remark}

\renewcommand{\i}{\operatorname{i}}

\newcommand\C{{\mathbb C}}
\newcommand\R{{\mathbb R}}

\newcommand{\ov}[1]{\overline{ #1}}

\newcommand{\lb}{[\cdot,\cdot]}

\newcommand{\N}{\nabla}

\newcommand{\bk}{\overline{k}}

\newcommand{\br}{\overline{r}}
\newcommand{\bs}{\overline{s}}

\newcommand{\g}{\mathfrak{g}}

\begin{document}
\title{A note on Canonical Ricci forms on $2$-step nilmanifolds}
\author{Luigi Vezzoni}
\date{\today}
\address{Dipartimento di Matematica, Universit\`a di Torino, Torino, Italy.}
\email{luigi.vezzoni@unito.it}

\thanks{The author was supported by the Project M.I.U.R.
``Riemannian Metrics and  Differentiable Manifolds''
and by G.N.S.A.G.A. of I.N.d.A.M.}
\maketitle
\begin{abstract}
In this note we prove that any left-invariant almost Hermitian structure on a $2$-step nilmanifold is Ricci-flat with respect to the Chern connection
and that it is Ricci -flat with respect to another canonical connection if and only if it is cosymplectic (i.e. $d^*\omega=0$).
\end{abstract}

\section{introduction}
Let $(M,g,J,\omega)$ be an almost Hermitian manifold. Gauduchon introduced in \cite{gau} a $1$-parameter family $\N^t$ of canonical Hermitian
connections  which can be distinguished by the properties of the torsion tensor $T$.
In this family $\N^1$ corresponds to so-called Chern connection which can be defined as the unique Hermitian connection whose $(1,1)$-part of the torsion vanishes. In the \emph{quasi-K\"ahler} case (i.e. when $\ov \partial \omega=0$), the line $\{\N^t\}$ degenerates to a single point and the Chern connection is the unique canonical connection.

Any canonical connection $\N^t$ induces the so-called \emph{Ricci form} $\rho^t(X,Y)=2\i {\rm tr}_{\omega}R^t(X,Y)$, where $R^t$ denotes the curvature of $\N^t$. It turns out that $\rho^t$ is always a closed form which can be locally written as the derivative of the $1$-form $\theta^t(X)=\sum_{r=1}^n g(\N^t_{X}Z_r,Z_{\br})$, where $\{Z_r\}$ is a (local) unitary frame. Moreover, in the cosymplectic case (i.e. when $d\omega^{n-1}=0$) the line $\{\theta^t\}$ degenerates to
a single point (see Corollary \ref{deg}) and all the canonical connections have the same Ricci form.

The aim of this paper is to study the Ricci forms $\rho^t$ on $2$-step nilmanifolds equipped with a left-invariant almost Hermitian structure. We recall
that by definition a \emph{$k$-step nilmanifold} is a compact quotient of a $k$-step nilpotent Lie group $G$ by lattice.
Since we are considering \emph{left-invariant} almost Hermitian structures, we can work on Lie algrebras in an algebraic fashion.
Our main result is the following

\begin{theo}\label{main}
Let $(\g,g,J,\omega)$ be a $2$-step nilpotent Lie algebra with an almost Hermitian structure.
Then $(g,J)$ is Ricci-flat with respect to
the Chern connection and it is Ricci-flat with respect to another canonical connection if and only if it is cosymplectic $($i.e. $d^*\omega=0)$.
\end{theo}
This theorem has the following immediate consequence:
\begin{cor}
Every left-invariant almost Hermitian structure on a nilmanifold associated to a $2$-step Lie group is Ricci-flat with respect to the Chern connection.
\end{cor}

\bigbreak\noindent{\it Acknowledgments.} The research of this paper has been motivated by a conversation with Simon Salamon. I'm very grateful to him. Furthermore, I'm grateful to Gueo Grantcharov for useful conversations and remarks and to Nicola Enrietti for an important observation on the presentation of the main result.

\section{Preliminaries on canonical connections}
Let $(M,g,J,\omega)$ be an almost Hermitian manifold, where $\omega$ is the fundamental form $\omega(\cdot,\cdot)=g(J\cdot,\cdot)\,.$
The almost complex structure $J$ extends to $r$-forms as
$$
J\alpha(X_1,\dots,X_n)=(-1)^r \alpha(JX_1,\dots,JX_n)
$$
inducing the splittings
$$
TM\otimes \C=T^{1,0}M\oplus T^{0,1}M\,,\quad \Lambda^r(M,\C)=\bigoplus_{p+q=r}\Lambda^{p,q}M\,,
$$
where $\Lambda^{r}(M,\C)$ is the vector bundle of complex $r$-forms on $M$. In particular $\Lambda^{3}M$ splits as
$$
\Lambda^{3}M=\Lambda^{+}M\oplus\Lambda^{-}M\,,
$$
where $\Lambda^{+}M=(\Lambda^{2,1}M\oplus\Lambda^{1,2}M)\cap \Lambda^{3}M$ and
$\Lambda^{-}M=(\Lambda^{3,0}M\oplus\Lambda^{0,3}M)\cap \Lambda^{3}M$. Given a $3$-form $\gamma$ we denote by
$\gamma^+$ and $\gamma^-$ the projection onto $\Lambda^{+}M$ and $\Lambda^{-}M$, respectively. Moreover,
denoting by $\Omega^2(TM)$ the vector space of smooth sections of $\Lambda^2M\otimes TM$, we have the splitting
$$
\Omega^2(TM)=\Omega^{2,0}(TM)\oplus \Omega^{1,1}(TM) \oplus \Omega^{0,2}(TM)
$$
where
$$
\begin{aligned}
& \Omega^{2,0}(TM)=\{B\in\Omega^2(TM)\,\,:\,\,B(JX,Y)=JB(X,Y) \}\,;\\
\vspace{0.1cm}
& \Omega^{1,1}(TM)=\{B\in\Omega^2(TM)\,\,:\,\,B(JX,JY)=B(X,Y) \}\,;\\
\vspace{0.1cm}
& \Omega^{0,2}(TM)=\{B\in\Omega^2(TM)\,\,:\,\,B(JX,Y)=-JB(X,Y) \}\,.
\end{aligned}
$$
Hence any $B\in\Omega^{2}(TM)$ can be written as $B=B^{2,0}+B^{1,1}+B^{0,2}$. Notice that in terms of complex vector fields of type $(1,0)$ we have
$$
B^{2,0}(Z_i,Z_j)=B(Z_i,Z_j)+B(Z_i,Z_j)-\i JB(Z_i,Z_j)-\i JB(Z_i,Z_j)=2 B(Z_i,Z_j)-2\i JB(Z_i,Z_j)\,.
$$
In particular the condition $B^{2,0}=0$ can be written in terms of $(1,0)$ vector fields as $B(Z_i,Z_j)\in T^{0,1}M\,.$
Furthermore $\Omega^2(TM)$ splits as
$$
\Omega^2(TM)=\Omega^2_b(TM)\oplus \Omega^2_c(TM)
$$
where
$$
\begin{aligned}
&g(B_b(X,Y),Z)=\frac12 (g(B(X,Y),Z)-g(B(Z,X),Y)-g(B(Y,Z),X))\,,\\
&g(B_c(X,Y),Z)=\frac12 (g(B(X,Y),Z)+g(B(Z,X),Y)+g(B(Y,Z),X))\,.
\end{aligned}
$$
Now we consider connections on $M$. A connection $\N$ on $M$ is called \emph{Hermitian} if $\N J=0$, $\N g=0$.
It is well-known that every almost Hermitian manifold admits Hermitian connections.
We denote by $\mathcal{C}$ the space of Hermitian connection on $M$. Gauduchon introduced in \cite{gau} the following special class of
Hermitian connections:
\begin{defn}
A connection $\N\in\mathcal{C}$ is called \emph{canonical} if its torsion $T$ satisfies $\,\,T_{b}^{1,1}=0\,.$
\end{defn}
\noindent From \cite{gau} it follows that any canonical connection $\N$ can be written as
\begin{equation}
\label{nablat}
\begin{aligned}
g(\N_{X}Y,Z) = &\,g(D_{X}Y,Z)+\frac{t-1}{4}(d^c\omega)^+(X,Y,Z)+\frac{t+1}{4}(d^c\omega)^+(X, JY, JZ)-g(X,N(Y,Z))+\\
                 &\,\frac12 (d^c \omega)^- (X,Y,Z)\,.
\end{aligned}
\end{equation}
for some $t\in\R$, where $d^{c}$ is the operator acting on $r$-forms as $d^c=(-1)^rJdJ$  and $N$ denotes the Nijenhuis tensor $N(X,Y)=[JX,JY]-[X,Y]-J([JX,Y]+[X,JY]).$

 For $t\in\R$ we denote by $\N^t$ the corresponding canonical connection.
In the special case of a quasi-K\"ahler structure (i.e. $\ov \partial \omega=0$) the space of canonical connections
reduces to a single point, while if
$J$ is integrable (i.e. $N=0$) equation  \eqref{nablat} reduces to
$$
g(\N^t_{X}Y,Z) = g(D_{X}Y,Z)+\frac{t-1}{4}(d^c\omega)(X,Y,Z)+\frac{t+1}{4}(d^c\omega)(X, JY, JZ)\,.
$$
For the parameters $t=1,0,-1$, the family \eqref{nablat} gives the following remarkable cases
\begin{itemize}
\item $t=1$. In this case $\nabla^1$ is called the \emph{Chern connetion}.
This connection can be defined as the unique Hermitian connection satisfying $T^{1,1}=0$.

\vspace{0.1cm}
\item $t=0$. In this case $\N^0$ is called the \emph{first canonical connection}. This connection can be defined
as the unique Hermitian connection whose torsion satisfies $T^{2,0}=0$.

\vspace{0.1cm}
\item $t=-1$. In this case the connection $\N^{-1}$ is important in the complex case where it is known as the \emph{Bismut connection}. Indeed, if $J$ is integrable, then $\N^{-1}$ can be defined as the unique Hermitian connection having totally skew-symmetric torsion (see \cite{Bismut}).

\end{itemize}
\section{Canonical Ricci forms}
Let $(M^{2n},g,J)$ be an almost Hermitian manifold and let $\mathcal{C}$ the space of the associated Hermitian connections.
For any $\N\in\mathcal{C}$ it is defined the Ricci form
$
\rho(X,Y)=2\i {\rm tr}_{\omega} R(X,Y)\,,
$
where $R$ is the curvature tensor
$R(X,Y):=[\N_X,\N_Y]-\N_{[X,Y]}$.  Such a form is always closed and it locally satisfies $\rho=d\theta$, where
$
\theta(X)=\sum_{r=1}^n g(\N_{X}Z_r,Z_{\br})
$
and $\{Z_r\}$ is a local unitary frame.
In the case of a canonical connection $\N^t\in\mathcal{C}$ we use notation  $\rho^t$ and $\theta^t$. We denote by $\natural$ the natural isomorphism between vector fields and $1$-forms induced by $g$. Namely, if $X$ a vector field, then we denote
by $X^{\natural}$ the $1$-form $X^{\natural}(Y)=g(X,Y)$.
We have the following

\begin{prop}\label{theta}
$\theta^t$ is locally defined by
\begin{equation}\label{thetat}
\theta^t(X)=\sum_{r=1}^n \i \Im\mathfrak{m}\left\{g([X +t\i JX,Z_r],Z_{\br})\right\}+\frac12\i (t-1) g(d^*\omega,X^{\natural})\,.
\end{equation}
for any vector field $X$.
\end{prop}
\begin{proof}
First of all we note that if $Z_r$ is a vector field of type $(1,0)$, then
$$
N(Z_r,Z_{\ov r})=(d^c\omega)^-(X,Z_{r},Z_{\ov r})=0\,,\quad (d^{c}\omega)^+(X,Z_{r},Z_{\ov r})=d^{c}\omega(X,Z_{r},Z_{\ov r})=-d\omega(JX,Z_{r},Z_{\ov r})\,.
$$
Hence if $\{Z_r\}$ is a local unitary frame using equation \eqref{nablat} we get
$$
\begin{aligned}
\theta^t(X)=& \sum_{r=1}^n \Big\{g(D_{X}Z_r,Z_{\br})+\frac{t-1}{4}(d^c\omega)(X,Z_r,Z_{\br})+\frac{t+1}{4}(d^c\omega)(X,JZ_r,JZ_{\br})\big\}\\
          =&\sum_{r=1}^n \Big\{g(D_{X}Z_r,Z_{\br})-\frac{t}{2}d\omega(JX,Z_r,Z_{\br})\Big\}\,.
\end{aligned}
$$
Now
$$
\begin{aligned}
2g(D_{X}Z_r,Z_{\br})=&Xg(Z_r,Z_{\ov r})-Z_{\br}g(X,Z_r)+Z_rg(X,Z_{\br})+g([X,Z_r],Z_{\br})+ g([Z_{\br},X],Z_r)-g([Z_r,Z_{\br}],X)\\
                    =&-Z_{\br}g(X,Z_r)+Z_rg(X,Z_{\br})+g([X,Z_r],Z_{\br})+ g([Z_{\br},X],Z_r)-g([Z_r,Z_{\br}],X)
\end{aligned}
$$
and
$$
\begin{aligned}
d\omega(JX,Z_r,Z_{\br})=&(JX)\omega(Z_r,Z_{\ov r})-Z_r\omega(JX,Z_{\ov r})+Z_{\ov r}\omega (JX,Z_r)-\omega([JX,Z_r],Z_{\br})\\
                        &-\omega([Z_{\br},JX],Z_{r})-\omega([Z_{r},Z_{\br}],JX)\\[5pt]
                       =&Z_r g(X,Z_{\ov r})-Z_{\ov r}g (X,Z_r)+\i g([JX,Z_r],Z_{\br})+\i g([Z_{\br},JX],Z_{r})-g([Z_{r},Z_{\br}],X)\,.
\end{aligned}
$$
Then we have
$$
\begin{aligned}
\theta^t(X)=&\frac12 \sum_{r=1}^n\Big\{g([X +t\i JX,Z_r],Z_{\br})- g([X-t\i JX,Z_{\br}],Z_r)+g([Z_r,Z_{\br}],tX-X)\\
            &+(1-t)Z_{r}g(X,Z_{\br})-(1-t)Z_{\br}g(X,Z_r)\Big\}\\
            =&\sum_{r=1}^n\Big\{\i \Im\mathfrak{m}\left\{g([X +t\i JX,Z_r],Z_{\br})+(1-t)Z_{r}g(X,Z_{\br})\right\}-\frac12(1-t) g([Z_r,Z_{\br}],X)\Big\}\,.
\end{aligned}
$$
So in order to prove the statement we have to show that
\begin{equation}\label{[Zr,Zbarr]}
\sum_{r=1}^n\Big\{\Im\mathfrak{m}\left\{Z_{r}g(X,Z_{\br})\right\}+\i \frac12g([Z_r,Z_{\br}],X)\Big\}=-\frac12\, g(\omega,dX^{\natural})\,.
\end{equation}
We can write $X=\sum_{r=1}^{n}(X_r Z_r+X_{\ov r}Z_{\br})$ and
$X^{\natural}=\sum_{r=1}^{n}(X_r\zeta^r+X_{\br}\zeta^{\ov r})$, where $\{\zeta^r\}$ is the coframe dual to
of $\{Z_r\}$. Then we get
$$
\begin{aligned}
g(\omega,dX^{\natural})=&\,\i\sum_{k=1}^ng(\zeta^{k}\wedge\zeta^{\bk},dX^{\natural})\\
=&\,\i\sum_{k,r=1}^n(Z_r(X_{\br})-Z_{\br}(X_{r})-X^{\natural}([Z_r,Z_{\ov r}]))g(\zeta^{k}\wedge\zeta^{\bk},\zeta^{r}\wedge\zeta^{\br})\\
=&\,\i\sum_{k=1}^nZ_k(X_{\bk})-Z_{\bk}(X_{k})-X^{\natural}([Z_k,Z_{\ov k}]))\\
=&\,-2\sum_{k=1}^n\Im\mathfrak{m}\{Z_k(X_{\bk})\}-\i\sum_{k,s=1}^n(B_{k\bk}^sX_s+B_{k\bk}^{\bs}X_{\bs})\\
=&\,-\sum_{k=1}^n \,\left(2\,\Im\mathfrak{m}\{Z_k(X_{\bk})\}+\i g([Z_k,Z_{\ov k}],X)\right)\,,
\end{aligned}
$$
where with $B$ we denote the components of the brackets.
\end{proof}
\begin{cor}
The following formulae hold
\begin{itemize}
\item $\theta^1(X)=2\i\sum_{r=1}^n \Im\mathfrak{m}\,g([X^{0,1},Z_r],Z_{\br})\,;$

\vspace{0.1cm}
\item $\theta^0(X)=\i\sum_{r=1}^n \Im\mathfrak{m}\left\{g([X,Z_r],Z_{\br})\right\}-\i\frac12 g(d^*\omega,X^{\natural})\,;$

\vspace{0.1cm}
\item $\theta^{-1}(X)=2 \i\sum_{r=1}^n \Im\mathfrak{m}\left\{g([X^{1,0},Z_r],Z_{\br})\right\}-\i g(d^*\omega,X^{\natural})\,.$
\end{itemize}

\end{cor}
It is useful to write down formula \eqref{thetat} in real coordinates. In order to do this we write
$Z_r=\frac{1}{\sqrt{2}}(e_r-\i Je_r)$ for a suitable orthonormal frame $\{e_1,\dots,e_n,Je_1,\dots,Je_n\}$. Then a direct computation gives
$$
\begin{aligned}
2\Im\mathfrak{m}\left\{ g([X +t\i JX,Z_r],Z_{\br})\right\}&= \Im\mathfrak{m}\left\{g([X +t\i JX,e_r-\i Je_r],e_r+\i Je_r)\right\}\\
&= g([X,e_r], Je_r)-g([X,Je_r],e_r)+tg([JX,e_r],e_r)+tg([JX,Je_r],Je_r)
\end{aligned}
$$
and
\begin{equation}
\begin{aligned}
\theta^t(X)=&\,\frac12 \i \sum_{r=1}^n \left\{g([X,e_r], Je_r)-g([X,Je_r],e_r)+tg([JX,e_r],e_r)+tg([JX,Je_r],Je_r)\right\}\\
            &\,+\frac12 \i(t-1) g(d^*\omega,X^{\natural})\,.
\end{aligned}
\end{equation}
A remarkable consequence of formula \eqref{thetat} is the following
\begin{cor}\label{deg}
All canonical connections of a cosymplectic structure have the same Ricci form.
\end{cor}
\begin{proof}
It is enough to show that $\theta^1=\theta^{-1}$. Since the cosymplectic
condition $d^*\omega=0$ implies
$$
\theta^{-1}(X)=\sum_{r=1}^n2 \i \Im\mathfrak{m}\left\{g([X^{1,0},Z_r],Z_{\br})\right\}
$$
we have
$$
\begin{aligned}
\theta^{-1}(X)=\,&-\sum_{r=1}^n2 \i \Im\mathfrak{m}\left\{g([X^{0,1},Z_{\br}],Z_{r})\right\}=
               -\sum_{r=1}^n2 \i \Im\mathfrak{m}\left\{g(D_{X^{0,1}}Z_{\br},Z_{r})-g(D_{Z_{\br}}X^{0,1},Z_{r})\right\}\\
              =\,& \sum_{r=1}^n2 \i\Im\mathfrak{m}\left\{g(D_{X^{0,1}}Z_{r},Z_{\br})-g(D_{Z_{\br}}X^{0,1},Z_{r})\right\}\\
              =\,&\theta^1(X)+\sum_{r=1}^n2 \i\Im\mathfrak{m}\left\{g(D_{Z_r}X^{0,1},Z_{\br})-g(D_{Z_{\br}}X^{0,1},Z_{r})\right\}\,.
\end{aligned}
$$
Now we observe that $\sum g(D_{Z_r}X^{0,1},Z_{\br})=-\sum g(X^{0,1},D_{Z_r}Z_{\br})=0$, since the cosymplectic
condition forces $\sum D_{Z_{\br}}Z_r$ to be of type $(1,0)$ (see e.g. \cite{Wood}). The last step consists to
show that $\sum \Im\mathfrak{m}\left\{ g(D_{Z_{\br}}X^{0,1},Z_{r})\right\}=0$. Here it is enough to consider the identity
$$
\sum_{r=1}^n\Big\{\Im\mathfrak{m}\left\{Z_{r}g(X,Z_{\br})\right\}+\i \frac12g([Z_r,Z_{\br}],X)\Big\}=\sum_{r=1}^n \Im\mathfrak{m}\left\{ g(D_{Z_{\br}}X^{0,1},Z_{r})\right\}
$$
which can be checked performing a direct computation. Then equation \eqref{[Zr,Zbarr]} implies the statement.
\end{proof}
\begin{rem}{\em
In the Hermitian case this last result was already known. In fact, it can be deduced from formula (8) of \cite{graD}. Another proof of this fact can be found in \cite{lui}.}
\end{rem}
\section{Canonical Ricci forms on Lie algebras}
Now we restrict our attention to left-invariant almost Hermitian structures on Lie groups (or more generally on
left-invariant almost Hermitian structures on quotient of Lie groups by lattices).
Since here all the computations are purely algebraic, we may assume to work on
a Lie algebra $(\g,\lb)$ equipped with an almost Hermitian structure $(g,J)$. An almost Hermitian structure on a Lie algebra is a pair $(g,J)$, where $J$ is an endomorphism of $\g$ satisfying $J^2=-{\rm Id}$ and $g$ is a $J$-Hermitian inner product.
The bracket of $\g$ has not a priori any relation with $J$. The pair $(g,J)$ induces as usual the fundamental form $\omega(\cdot,\cdot)=g(J\cdot,\cdot)$.

\medskip
Proposition \ref{theta} implies the following
\begin{prop}\label{g}
Let $(\g,\lb,g,J)$ be a Lie algebra with an almost Hermitian structure.
For any $t\in\R$ the following formula holds
\begin{equation}
\label{thetatinv}
\theta^t(X)=\frac12 \i\left\{-{\rm tr}({\rm ad}_{X}\circ J)+ t\,{\rm tr}\,{\rm ad}_{JX}+(t-1)\, g(d^*\omega,X^{\natural})  \right\}\,.
\end{equation}
Moreover if  $(\g,\lb)$ is unimodular $($i.e. ${\rm tr}\,{\rm ad}_X=0$ for any $X\in \g);$ then
\begin{equation}\label{tethaunimodular}
\rho^t(X,Y)=\frac12 \i {\rm tr}({\rm ad}_{[X,Y]}\circ J)-\frac12\i(t-1)\, g(d^*\omega,[X,Y]^{\natural})
\end{equation}
and $\rho^t$ is the same for any $t$ if and only if $(g,J)$ cosymplectic.
\end{prop}
\begin{proof}
The only non-trivial part of the statement is the last assertion. So
we have just to show that condition $g(\omega,d[X,Y]^{\natural})=0$ is equivalent to $d^*\omega=0$.
We can write $\g=[\g,\g]\oplus [\g,\g]^{\perp}$. Let $X\in [\g,\g]^{\perp}$, then
$$
dX^{\natural}(Z,W)=-g(X,[Z,W])=0\,.
$$
Hence for every $X\in[\g,\g]^{\perp}$, $dX^{\natural}=0$. This implies that $d\g^*=d([\g,\g]^{\natural})$ and the claim follows.
\end{proof}

Now we can prove Theorem \ref{main}:
\begin{proof}[Proof of Theorem $\ref{main}$]
Let $(\g,\lb,g,J)$ is a $2$-step nilpotent Lie algebra with an almost
Hermitian structure. Then, taking into account that $\g$ is unimodular, the {\rm 2-step} condition implies that $[\g,\g]$ is contained in the center of $\g$ and ${\rm tr}({\rm ad}_{[X,Y]}\circ J)=0$ for every $X,Y\in \g$.
Then formula \eqref{tethaunimodular} reduces to
\begin{equation}\label{rhot}
\rho^t(X,Y)=\frac12\i (1-t)\,g(d^*\omega,[X,Y]^{\natural})
\end{equation}
and first claim follows.
\end{proof}



Proposition \ref{g} allows us to describe the behavior of $\{ \rho^t\}$ for some special almost Hermitian structures:
\begin{prop}
Let $(\g\lb,g,J)$ be an almost Hermitian Lie algebra.
\begin{itemize}
\item If $J$ is bi-invariant $($i.e. $[J\cdot,\cdot]=J[\cdot,\cdot])$, then
$$
\theta^{t}(X)=(t-1)\i{\rm tr}({\rm ad}_{JX})\,,\quad \rho^t(X,Y)=\i(1-t)\,{\rm tr}({\rm ad}_{[JX,Y]})\,.
$$

\item If $J$ is anti-bi-invariant $($i.e. $[J\cdot,\cdot]=-J[\cdot,\cdot])$, then
$$
\theta^{t}=0\,,\quad \rho^t=0\,.
$$

\item If $J$ is abelian $($i.e. $[J\cdot,J\cdot]=[\cdot,\cdot]$ $)$, then
$$
\begin{aligned}
& \theta^t(X)=\frac12\i\left\{ (1+t)\,{\rm tr}({\rm ad}_{JX})+(t-1)\, g(d^*\omega,X^\natural)\right\}\,,\\
& \rho^t(X)=\frac12\i\left\{- (1+t)\,{\rm tr}({\rm ad}_{J[X,Y]})+ (1-t)\,g(d^*\omega,[X,Y]^\natural)\right\}\,.
\end{aligned}
$$
\item If $J$ is anti-abelian $($i.e. $[J\cdot,J\cdot]=-[\cdot,\cdot]$ $)$, then
$$
\theta^t(X)=\frac12 \i (1+t)\,{\rm tr}({\rm ad}_{JX})\,,\quad
\rho^t(X,Y)=-\frac12 \i (1+t)\,{\rm tr}({\rm ad}_{J[X,Y]})\,.
$$
\end{itemize}
In particular in the unimodular case bi-invariant, anti-bi-invariant and anti-abelian almost Hermitian structures are Ricci-flat
with respect to any canonical connection, while in the abelian case $\rho^t$ is given by the following formula
$$
\rho^t(X,Y)=\frac12\i(1-t)\,g(d^*\omega,[X,Y]^\natural).
$$
and $\rho^t=0$ for $t\neq 0$ if and only if $(g,J)$ is a cosymplectic structure.
\end{prop}

\begin{rem}
{\em We remark the following facts:
\begin{itemize}
\item The bi-invariant condition $[J\cdot,\cdot]=J[\cdot,\cdot]$ is equivalent to require that the simply-connected  Lie group associated to $(\g,J)$ is a complex Lie group. The fact that a bi-invariant almost Hermitian structure on an unimodular Lie algebra is Ricci-flat with respect any canonical connection has been already proved by Grantcharov in \cite{Gueo}.

\vspace{0.1cm}
\item The anti-bi-invariant condition $[J\cdot,\cdot]=-J[\cdot,\cdot]$ is equivalent to require that any $J$-compatible
inner product on $\g$ is quasi-K\"ahler and flat with respect to the Chern connection $\N^1$ (see \cite{DV}).

\vspace{0.1cm}
\item The abelian condition $[J\cdot,J\cdot]=[\cdot,\cdot]$  was introduced in \cite{Barberis} and was intensely studied in \cite{Adrian,BarberisDottiVerbisky,Dotti Fino,sergiun,Maclaughlin}. This condition is equivalent to require that $\g^{1,0}$ is an abelian Lie algebra.

\vspace{0.1cm}
\item Finally, the  anti-abelian condition $[J\cdot,J\cdot]=-[\cdot,\cdot]$ was studied in \cite{DVL}.
\end{itemize}
}
\end{rem}

\begin{rem}{\em
Theorem \ref{main} can be applied to the Heisenberg Lie algebras $\mathfrak{h}_{n}(\R)$ and $\mathfrak{h}_n(\C)$.
That accords to Theorem 4.1 of \cite{tosatti} and Proposition 4.10 and 4.11 of \cite{DVagag}. Moreover things work differently either in  the $3$-step nilpotent case or in the $2$-step solvable case (see \cite{DVagag}).
}
\end{rem}

\end{document}